\documentclass[reqno]{amsart}

\usepackage{times}
\usepackage{mathrsfs}
\usepackage{amsmath}
\usepackage{amsthm}
\usepackage{amsfonts}

\newtheorem{thm}{Theorem}[section]
\newtheorem{lem}[thm]{Lemma}

\newtheorem{rem}[thm]{Remark}

\DeclareMathAlphabet{\mathpzc}{OT1}{pzc}{m}{it}

\numberwithin{equation}{section}

\newcommand{\Wqb}{W_{q,D}}
\newcommand{\R}{\mathbb{R}}

\newcommand{\A}{\mathbb{A}}

\newcommand{\E}{\mathbb{E}}

\newcommand{\ml}{\mathcal{L}}

\newcommand{\Om}{\Omega}

\newcommand{\ve}{\varepsilon}
\newcommand{\rd}{\mathrm{d}}

\newcommand{\bqn}{\begin{equation}}
\newcommand{\eqn}{\end{equation}}
\newcommand{\bqnn}{\begin{equation*}}
\newcommand{\eqnn}{\end{equation*}}
\newcommand{\bear}{\begin{eqnarray}} 
\newcommand{\eear}{\end{eqnarray}} 
\newcommand{\bean}{\begin{eqnarray*}} 
\newcommand{\eean}{\end{eqnarray*}} 
\newcommand{\bs}{\begin{split}}
\newcommand{\es}{\end{split}}

\newcommand{\dhr}{\mathrel{\lhook\joinrel\relbar\kern-.8ex\joinrel\lhook\joinrel\rightarrow}}

\title[Coexistence Steady States in a Predator-Prey Model]
{Coexistence Steady States in a Predator-Prey Model}

\author[Ch. Walker]{Christoph Walker}

\address{Leibniz Universit\"at Hannover, Institut f\"ur Angewandte Mathematik, Welfengarten 1, D--30167 Hannover, Germany.}
\email{walker@ifam.uni-hannover.de}

\begin{document}

\begin{abstract}
An age-structured predator-prey system with diffusion and Holling-Tanner-type nonlinearities is considered. Regarding the intensity of the fertility of the predator as bifurcation parameter, we prove that a branch of positive coexistence steady states bifurcates from the marginal steady state with no prey. A similar result is obtained when the fertility of the prey varies.
\end{abstract}


\maketitle

\section{Introduction}
We consider the situation that an age-structured prey population and an age-structured predator population inhabit the same region. If $u=u(t,a,x)\ge 0$ and $v=v(t,a,x)\ge 0$ are respectively the density functions of the prey and predator at time $t\ge 0$, age $a\in [0,a_m)$, and spatial position $x\in\Om$, a general model of equations governing the time evolution reads
\begin{align*}
\partial_tu+\partial_a u-d_1\Delta_x u&=-\mu_1(a,u,v)u\ ,& t>0\ ,\quad a\in (0,a_m)\ ,\quad x\in\Om\ ,\\
u(t,0,x)&=\int_0^{a_m}b_1(a,u,v)\, u(t,a,x)\rd a\ , & t>0\ ,\quad x\in\Om\ ,\\
\partial_tv+\partial_a v-d_2\Delta_x v&=-\mu_2(a,u,v)v\ ,& t>0\ ,\quad a\in (0,a_m)\ ,\quad x\in\Om\ ,\\
v(t,0,x)&=\int_0^{a_m}b_2(a,u,v)\, v(t,a,x)\rd a\ , & t>0\ ,\quad x\in\Om\ ,
\end{align*}
subject to some suitable boundary conditions on the boundary $\partial\Om$. Here, $\mu_j$ and $b_j$ are respectively the death and birth rates depending nonlinearly on the predator $v$ and on the prey $u$, $\Om\subset\R^n$ is a bounded and smooth domain, and $a_m\in (0,\infty]$ is the maximal age (that could be the different for the two populations). In this paper, however, we shall focus on steady state solutions, that is, on time-independent solutions $u=u(a,x)\ge 0$ and $v=v(a,x)\ge 0$, for a particular case of the previous equations. More precisely, we look for nonnegative solutions $(u,v)$ to the parameter-dependent system
\begin{align}
\partial_a u-\Delta_x u&=-\alpha_1Uu-\alpha_2\frac{Vu}{1+mU}\ ,&  a\in (0,\infty)\ ,\quad x\in\Om\ ,\label{4a}\\
u(0,x)&=\eta U(x)\ , &  x\in\Om\ ,\label{4b}\\
\partial_a v-\Delta_x v&=-\beta_1Vv+\beta_2\frac{Uv}{1+mU}\ ,&  a\in (0,\infty)\ ,\quad x\in\Om\ ,\label{5a}\\
v(0,x)&=\xi V(x)\ , &  x\in\Om\ ,\label{5b}
\end{align}
where
\bqn\label{3}
U:=\int_0^\infty e^{-ra}\, u(a,\cdot)\,\rd a\ ,\qquad  V:=\int_0^\infty e^{-sa}\,v(a,\cdot)\,\rd a\ .
\eqn
Equations \eqref{4a}, \eqref{5a} are supplemented with Dirichlet boundary conditions, i.e. $u\vert_{\partial\Om}=0$ and $v\vert_{\partial\Om}=0$. The latter system is derived from the previous one by taking $a_m:=\infty$, by normalizing the diffusivities $d_1, d_2$ to 1 for the sake of readability, by considering linear birth rates of the form 
$$
b_1(a):=\eta e^{-ra}\ ,\qquad b_2(a):=\xi e^{-sa}\ ,
$$ 
where $\eta>0$ and $\xi>0$ are parameters measuring the intensity of the fertility and $r,s>0$ are weights for the loss of fertility with increasing age, and by taking
nonlinear mortality rates of the form
$$
\mu_1(u,v):=\alpha_1U+\alpha_2\frac{V}{1+mU}\ ,\qquad \mu_2(u,v):=\beta_1V-\beta_2\frac{U}{1+mU}
$$
with some fixed constants $\alpha_1,\alpha_2,\beta_1,\beta_2,m>0$. 

Clearly, other boundary conditions, e.g. of Neumann type, can be considered as well. We point out that equations \eqref{4a}-\eqref{5b} are nonlocal with respect to age due to the nonlinear terms involving $U$ and $V$ given in \eqref{3}. In addition, the initial values depend on the entire solution.

A formal integration of the parabolic system \eqref{4a}-\eqref{3} yields a nonlinear elliptic system for $(U,V)$:
\begin{align}
-\Delta_x U&=(\eta-r)U-\alpha_1U^2-\alpha_2\frac{VU}{1+mU}\ ,\quad  x\in\Om\ ,\qquad U\vert_{\partial\Om}=0\ ,\label{6}\\
-\Delta_x V&=(\xi-s)V-\beta_1V^2+\beta_2\frac{UV}{1+mU}\ ,\quad   x\in\Om\ ,\qquad V\vert_{\partial\Om}=0\ .\label{7}
\end{align}
Note that with time dependence in \eqref{6}, \eqref{7} (and also in \eqref{4a}-\eqref{5b}), in the absence of the other specie and of diffusion, both species would grow logistically. The additional nonlinear coupling terms are referred to as Holling-Tanner reaction terms and represent, e.g. in \eqref{6} and \eqref{4a}, the rate at which the prey is consumed by the predator. This rate is finite even if the prey tends to infinity, i.e. reaction terms of Holling-Tanner type model e.g. finite appetite of the predator.

System \eqref{6}, \eqref{7} is investigated in \cite{BlatBrown} and global bifurcation results are shown with respect to the parameters $\eta-r$ and $\xi-s$. The goal of this paper is to show similar -- though local -- bifurcation results with respect to the parameters $\eta$ and $\xi$ for the parabolic system \eqref{4a}-\eqref{5b} in the spirit of \cite{BlatBrown}. We also refer to \cite{DelgadoEtAl2}, where a variant of \eqref{4a}-\eqref{5b} with only one equation is studied.

Obviously, independent of what the parameters $\eta$ and $\xi$ are, equations \eqref{4a}-\eqref{5b} always possess the trivial solution $(u,v)\equiv (0,0)$. Moreover, it follows from \cite{WalkerSIMA} that \eqref{4a}, \eqref{4b} with $V\equiv 0$ have nontrivial nonnegative solutions $u\not\equiv 0$ provided the parameter $\eta$ is suitable. Analogously, \eqref{5a}, \eqref{5b} with $U\equiv 0$ admit nontrivial nonnegative solutions $v\not\equiv 0$ for some values of $\xi$. In this paper we shall prove that, in addition, there are nonnegative coexistence steady states $(u_*,v_*)$ with $u_*\not\equiv 0$ and $v_*\not\equiv 0$ for some parameter values of $\eta$ and $\xi$. Roughly speaking, if $\xi$ is regarded as bifurcation parameter and $(\eta,u_\eta)$ is a fixed nontrivial and nonnegative solution (i.e. $u_\eta\not\equiv 0$) to \eqref{4a}, \eqref{4b} with $V\equiv 0$, then there is a critical value $\xi_0=\xi_0(\eta)$ such that a branch of nonnegative solutions $(\xi,u_*,v_*)$ to \eqref{4a}-\eqref{5b} with $u_*\not\equiv 0$ and $v_*\not\equiv 0$ bifurcates locally from the semi-trivial branch $\{(\xi,u_\eta,0);\xi\ge 0\}$ at the point $(\xi_0,u_\eta,0)$ provided that $\beta_2<<m$. This bifurcation is supercritical. We refer to Theorem~\ref{T1} for details. Conversely, if $\eta$ is regarded as bifurcation parameter, then a similar result can be derived without additional assumptions on the coefficients. The precise statement for this case is given in Theorem~\ref{T2}. 

In the next section we prove Theorem~\ref{T1} in detail using the theorem of Crandall-Rabinowitz \cite{CrandallRabinowitz}. The proof of Theorem~\ref{T2} is basically the same and will thus merely be sketched.

\section{Nontrivial Coexistence Steady States}

If $E$ and $F$ are Banach spaces we write $\ml(E,F)$ for the space of all bounded linear operators from $E$ to $F$, and we set $\ml(E):=\ml(E,E)$. 

We begin with some preliminary investigations. Fix $q\in (n+2,\infty)$ and let $$
\Wqb^\kappa:=\Wqb^\kappa(\Om):=\{u\in W_q^\kappa; u=0\ \text{on}\ \partial\Om\}
$$
denote the Sobolev-Slobodeckii spaces on $\Om$ involving Dirichlet boundary conditions for $\kappa>1/q$, where values on the boundary are interpreted in the sense of traces. Then $\Wqb^{2-2/q}\hookrightarrow C^1(\bar{\Om})$ by the Sobolev embedding theorem, in particular the interior of the positive cone  $\Wqb^{2-2/q}\cap L_q^+$ is nonempty. Set \mbox{$L_q:=L_q(\Om)$} and 
$$
\E_0:=L_q(\R^+,L_q)\ ,\quad \E_1:=L_q(\R^+,\Wqb^2)\cap W_q^1(\R^+,L_q)\ .
$$ 
For the positive cone of $\E_1$ we write $\E_1^+:=\E_1\cap L_q(\R^+,L_q^+)$. Recall that 
\bqn\label{emb}
\E_1\hookrightarrow BUC\big(\R^+,\Wqb^{2-2/q}\big)\hookrightarrow BUC\big(\R^+,C^1(\bar{\Om})\big)
\eqn 
according to \cite[III.Thm.4.10.2]{LQPP}. Hence the trace $\gamma_0u:=u(0)$ defines an operator \mbox{$\gamma_0\in\ml (\E_1,\Wqb^{2-2/q})$}. We then say that an operator $A\in\ml (\Wqb^2,L_q)$ has {\it maximal $L_q$-regularity on $\R^+$} provided that $$(\partial_a +A,\gamma_0)\in\ml (\E_1,\E_0\times \Wqb^{2-2/q})$$ is a toplinear isomorphism. 

Obviously, if $u \in\E_1$ and $\tau>0$, then $\int_0^\infty e^{-\tau a}\, u(a)\rd a\in\Wqb^2$ and, by \eqref{emb},
$$
\int_0^\infty e^{-\tau a}\, \partial_ a u(a)\,\rd a=-u(0)+\tau\int_0^\infty e^{-\tau a}\, u(a)\,\rd a\quad\text{in}\quad L_q\ .
$$
Throughout this paper we agree upon the notation \eqref{3} for $U$ and $V$ if $u,v\in\E_1$. 

We write $-\Delta_D$ for the Laplace operator subject to Dirichlet boundary conditions. It is known (e.g. \cite[Thm.12]{AmannPrinc}) that if $p\in L_\infty(\Om)$, then the eigenvalue problem
$$
-\Delta_D\varphi+ p\varphi=\lambda\varphi\ ,
$$
has a smallest eigenvalue $\lambda=\lambda_1(p)$ with a strongly positive eigenfunction. This principal eigenvalue $\lambda_1(p)$ is simple and increasing in $p$ \cite[Thm.16]{AmannPrinc}. We set $\lambda_1:=\lambda_1(0)>0$ and let $\varphi_1$ denote a strongly positive eigenfunction corresponding to $\lambda_1$.

The next lemma was noted in \cite{BlatBrown}.

\begin{lem}\label{L1}
Let $(u,v)$ be a nonnegative smooth solution to \eqref{4a}-\eqref{5b}. If $u\not\equiv 0$, then $\eta> \lambda_1+r$, and if $v\not\equiv 0$, then $\xi> \lambda_1+s-\beta_2/m$.
\end{lem}

\begin{proof}
Let $u\not\equiv 0$ and set $z(a):=\int_\Om \varphi_1 u(a)\rd x$. Then, since $\partial_a u-\Delta_D u\le 0$, we have $z'\le-\lambda_1 z$, i.e. $z(a)\le z(0)e^{-\lambda_1 a}$. Hence
$$
0\not= z(0)=\int_\Om\varphi_1\eta\int_0^\infty e^{-ra} u(a)\,\rd a\rd x <\dfrac{\eta}{\lambda_1+r} z(0)
$$
implies the first assertion. For the second claim let $v\not\equiv 0$ and set $w(a):=\int_\Om \varphi_1v(a)\rd x$. Then we obtain from
$\partial_a v-\Delta_D v\le \frac{\beta_2}{m}v$ that $w'\le(-\lambda_1 +\frac{\beta_2}{m})w$ and we conclude as before.
\end{proof}

Next, set $$\A_1(u):=-\Delta_D+\alpha_1U\quad\text{and}\quad \A_2(v):=-\Delta_D+\beta_1V$$ for $u,v\in\E_1$. Clearly, $\A_j\in C^1(\E_1,\ml(\Wqb^2,L_q))$ and $-\A_j(u)$ generates for each $u\in\E_1$ a strongly positive analytic semigroup $\{e^{-\A_j(u)a};a\ge 0\}$ on $L_q$. Moreover, $\A_j(0)=-\Delta_D$ has maximal $L_q$-regularity on $\R^+$ (e.g., see \cite[III.Ex.4.7.3,III.Thm.4.10.7]{LQPP}). We thus may apply the result of \cite{WalkerSIMA} to obtain semi-trivial branches of solutions to \eqref{4a}-\eqref{5b}, i.e. nontrivial solutions $(\xi,\eta,u,v)$ with either $u\equiv 0$ or $v\equiv 0$. In fact, we have:

\begin{lem}\label{L2}
(a) There are $\ve_0>0$ and a branch of nonnegative solutions $(\xi,v)$ to \eqref{5a}, \eqref{5b} with $U\equiv 0$ of the form
$$\mathcal{V}:=\{(\xi,v_\xi);\lambda_1+s<\xi<\lambda_1+s+\ve_0\}\subset\R^+\times\E_1^+$$ with $v_\xi\not\equiv 0$ bifurcating from the critical point $(\xi,v)=(\lambda_1+s,0)$.

(b) There are $\ve_0'>0$ and a branch of nonnegative solutions $(\eta,u)$ to \eqref{4a}, \eqref{4b} with $V\equiv 0$ of the form
$$\mathcal{U}:=\{(\eta,u_\eta);\lambda_1+r<\eta<\lambda_1+r+\ve_0'\}\subset\R^+\times\E_1^+$$ with $u_\eta\not\equiv 0$ bifurcating from the critical point $(\eta,u)=(\lambda_1+r,0)$.
\end{lem}

\begin{proof}
Of course, the proof of (a) and (b) is the same. We take $U\equiv 0$ in \eqref{5a} and apply \cite[Thm.2.4,Prop.2.8]{WalkerSIMA} to \eqref{5a}, \eqref{5b}, where we regard $\xi$ as bifurcation parameter. Observing that the compact and strongly positive operator $Q_0$ introduced in \cite{WalkerSIMA} is simply the resolvent
$$
Q_0:=\int_0^\infty e^{-sa}\,e^{\Delta_D a}\,\rd a=(s-\Delta_D)^{-1}\ ,
$$
we have 
$Q_0\varphi_1=
(s+\lambda_1)^{-1}\varphi_1$.
Hence, the spectral radius of $Q_0$ is $r(Q_0)=(s+\lambda_1)^{-1}$ since this is the only eigenvalue with a positive eigenfunction according to the Krein-Rutman theorem, and the existence of such a branch follows. Arguments similar to the proof of Lemma~\ref{L1} show that $\xi>\lambda_1+s$ for any nonnegative solution $(\xi,v)$ and so supercritical bifurcation occurs.
\end{proof}

Standard regularity theory for semilinear parabolic equations implies that the solutions of \eqref{4a}-\eqref{5b} established in Lemma~\ref{L2} are classical solutions, i.e. belong to $C(\R^+\times\bar{\Om})\cap C^{1,2}((0,\infty)\times\bar{\Om})$.

\subsection{Bifurcation for the Parameter $\xi$}

We first regard $\xi$ as bifurcation parameter and keep $\eta$ fixed. If $\eta\le \lambda_1+r$, then there is a trivial branch $\{(\xi,0,0);\xi\ge 0\}$ and a semi-trivial branch
$$
\mathfrak{C}:=\{(\xi,0,v_\xi)\,;\,\lambda_1+s<\xi<\lambda_1+s+\ve_0\}\subset\R^+\times\E_1^+\times\E_1^+
$$
of solutions $(\xi,u,v)$ to \eqref{4a}-\eqref{5b} provided by Lemma~\ref{L2}. If $\eta\in (\lambda_1+r,\lambda_1+r+\ve_0')$ and $(\eta,u_\eta)\in\mathcal{U}$, then Lemma~\ref{L2} ensures in addition the existence of another semi-trivial branch
$$
\mathfrak{C}_\eta:=\{(\xi,u_\eta,0)\,;\,\xi\ge 0\}\subset\R^+\times\E_1^+\times\E_1^+\ .
$$
Our aim is to show that under certain assumptions on the coefficients in \eqref{4a}-\eqref{5b}, a branch of positive coexistence steady states bifurcates from the branch $\mathfrak{C}_\eta$.

For the remainder of this subsection we fix $(\eta,u_\eta)\in\mathcal{U}$ and set $U_\eta:=\int_0^\infty e^{-ra}\,u_\eta(a)\rd a$. Note that 
$$
u_\eta(a)=\eta e^{(\Delta_D-\alpha_1U_\eta)a}U_\eta\ ,\quad a\ge 0\ ,
$$
and
\bqn\label{16}
-\Delta_DU_\eta=(\eta-r)U_\eta-\alpha_1U_\eta^2\ .
\eqn
The strong positivity of $e^{(\Delta_D-\alpha_1U_\eta)a}$ ensures $u_\eta(a)>0$ in $\Om$ for $a>0$, and $U_\eta$ is strongly positive. 
To shorten notation we set 
$$
p_\eta:=\dfrac{u_\eta}{1+mU_\eta}\quad\text{and}\quad P_\eta:=\dfrac{U_\eta}{1+mU_\eta}\ .
$$

\begin{lem}\label{L3}
We have $0<U_\eta(x)\le\dfrac{\eta-r}{\alpha_1}$ for $x\in\Om$. If $(\xi,u,v)$ is a nonnegative solution to \eqref{4a}-\eqref{5b}, then $0\le U(x)\le U_\eta(x)$ for $x\in\Om$. If $v\not\equiv 0$, then $\xi>\xi_0(\eta)$, where $\xi_0(\eta)$ is the principal eigenvalue of $-\Delta_D+s-\beta_2P_\eta$.
\end{lem}

\begin{proof}
The statement follows from \eqref{6}, \eqref{7}, \eqref{16}, and \cite[Lem.2.3,Lem.2.5,Thm.4.1]{BlatBrown}. We thus omit details and only sketch the simple proofs. Since $(\eta-r)/\alpha_1$ is a supersolution and $U$ a subsolution of \eqref{16}, the first and the second assertion follow. For the last assertion one multiplies the inequality 
$$
-\Delta_DV-\beta_2P_\eta V\le (\xi-s)V-\beta_1V^2\ \ ,
$$   
by $V$, integrates over $\Om$, and uses the fact that $\xi_0(\eta)-s$ is the principal eigenvalue of $-\Delta_D-\beta_2 P_\eta$.
\end{proof}

Note that the statement about the restriction of $\xi$ in Lemma~\ref{L3} is more precise than in Lemma~\ref{L1} due to $\beta_2P_\eta\le\beta_2/m$ and the fact that the principal values thus satisfy $\lambda_1(-\beta_2P_\eta)\ge \lambda_1(-\beta_2/m)$.

For future purposes let us also state the following auxiliary result:

\begin{lem}\label{L4}
The operator $-\Delta_D+\alpha_1U_\eta$ has maximal $L_q$-regularity on $\R^+$. If 
\bqn\label{122}
\dfrac{\beta_2(\eta-r)}{\alpha_1+m(\eta-r)}\quad\text{is sufficiently small}
\eqn (e.g., if $\beta_2/m$ is small), then also
$-\Delta_D-\beta_2P_\eta$ has maximal $L_q$-regularity on $\R^+$.
\end{lem}

\begin{proof}
Observing that $-\Delta_D+\alpha_1U_\eta$ has spectral bound not exceeding $-\lambda_1<0$ since $\alpha_1U_\eta$ is nonnegative, it follows from \cite[III.Ex.4.7.3,III.Thm.4.10.7]{LQPP} that $-\Delta_D+\alpha_1U_\eta$ has maximal $L_q$-regularity on $\R^+$. Analogously, due to $\lambda_1>0$ the Laplace operator $-\Delta_D$ has bounded imaginary power with power angle less than $\pi/2$ by \cite[III.Ex.4.7.3]{LQPP}. Moreover, using 
\bqn\label{P}
\left\|\beta_2P_\eta\right\|_\infty\le\frac{\beta_2(\eta-r)}{\alpha_1+m(\eta-r)} 
\eqn
by Lemma~\ref{L3}, we may invoke the perturbation theorem \cite[III.Thm.4.8.7]{LQPP} to conclude that $-\Delta_D-\beta_2P_\eta$ still has bounded imaginary power with power angle less than $\pi/2$ provided that the quotient on the right hand side of the previous inequality is sufficiently small. The assertion then follows again from
\cite[III.Thm.4.10.7]{LQPP}.
\end{proof}

Note that \cite[III.Thm.4.8.7]{LQPP} allows us in principle to compute the smallness condition in the statement of Lemma~\ref{L4} explicitly.
In the sequel we assume that $\beta_2(\eta-r)(\alpha_1+m(\eta-r))^{-1}$ is sufficiently small so that Lemma~\ref{L4} applies. In particular, we assume this number to be less than $\lambda_1+s$. Then $s-\beta_2P_\eta>-\lambda_1$ by \eqref{P} and thus
\bqn\label{A}
\xi_0:=\xi_0(\eta):=\lambda_1(s-\beta_2P_\eta)>0
\eqn
due to the monotonicity in $p$ of the principal eigenvalue $\lambda_1(p)$.

Suppose now that $(\xi,u,v)=(\xi,u_\eta-w,v)$ solves \eqref{4a}-\eqref{5b}. Then $(\xi,w,v)$ solves 
\begin{align}
\partial_a w-\Delta_D w&=-\alpha_1Wu_\eta-\alpha_1(U_\eta-W)w+\alpha_2\frac{V(u_\eta-w)}{1+m(U_\eta -W)}\ ,&w(0)=\eta W\ , \label{12}\\
\partial_a v-\Delta_D v&=-\beta_1Vv+\beta_2\frac{(U_\eta-W)v}{1+m(U_\eta-W)}\ ,&v(0)=\xi V\ , \label{13}
\end{align}
where
\bqnn
W:=\int_0^\infty e^{-ra}\, w(a)\,\rd a\ ,\qquad  V:=\int_0^\infty e^{-sa}\,v(a)\,\rd a\ .
\eqnn
Due to Lemma~\ref{L4} the operators
\begin{align*}
&Z_1:=\left(\partial_a-\Delta_D+\alpha_1U_\eta,\gamma_0\right)^{-1}\in\ml(\E_0\times \Wqb^{2-2/q},\E_1)\ ,\\
&Z_2:=\left(\partial_a-\Delta_D-\beta_2P_\eta,\gamma_0\right)^{-1}\in\ml(\E_0\times \Wqb^{2-2/q},\E_1)\ ,
\end{align*}
are well-defined. Hence, writing $\xi=\xi_0+t$, the solutions $(t,w,v)$ of \eqref{12}-\eqref{13} are the zeros of the function $F$ given by
$$
F(t,w,v):=\left(\begin{array}{c} w-Z_1\left(-\alpha_1W(u_\eta-w)+\frac{\alpha_2V(u_\eta-w)}{1+m(U_\eta -W)}\,,\,\eta W\right)\\
v-Z_2\left(-\beta_2P_\eta v-\beta_1Vv+\beta_2\frac{(U_\eta-W)v}{1+m(U_\eta-W)}\,,\, (\xi_0+t)V\right)\end{array}\right)\ .
$$
We validate the assumptions of the Crandall-Rabinowitz theorem \cite[Thm.1.7]{CrandallRabinowitz} for the function $F$. For $R>0$ sufficiently small set $\Sigma:=\mathbb{B}_{\E_1}(0,R)$ and note that
\bqnn\label{9}
\big[(u,v)\mapsto \frac{V}{1+mU}\big]\in C^1\big(\Sigma\times\Sigma,C(\bar{\Om})\big)\ ,
\eqnn
where we agree upon the notation \eqref{3}. Making $R>0$ smaller, if necessary, it readily follows that $F:\R\times\Sigma\times\Sigma\rightarrow\E_1\times\E_1$ has continuous partial Frech\'{e}t derivatives $F_t$, $F_{(w,v)}$, and $F_{t,(w,v)}$. Moreover, if $(\phi,\psi)\in\E_1\times\E_1$ and
\bqn\label{17}
\Phi:=\int_0^\infty e^{-ra}\, \phi(a)\,\rd a\ ,\qquad  \Psi:=\int_0^\infty e^{-sa}\,\psi(a)\,\rd a\ ,
\eqn
then the derivatives at $(t,w,v)=(0,0,0)$ are
\bqn\label{27}
F_{(w,v)}(0,0,0)[\phi,\psi]=\left(\begin{array}{c} \phi-Z_1\left(-\alpha_1\Phi u_\eta +\alpha_2\Psi p_\eta\,,\,\eta \Phi\right)\\
\psi-Z_2\left(0, \xi_0\Psi\right)\end{array}\right)\ 
\eqn
and
\bqn\label{26}
F_{t,(w,v)}(0,0,0)[\phi,\psi]=\left(\begin{array}{c} 0\\
-Z_2\left(0, \Psi\right)\end{array}\right)\ .
\eqn
Before analyzing $L:=F_{(w,v)}(0,0,0)\in\ml(\E_1\times\E_1,\E_1\times\E_1)$ further, let us observe, as in \cite{BlatBrown}, that the operator $$-\Delta_D+r-\eta+2\alpha_1 U_\eta\in\ml(\Wqb^2,L_q)$$ is invertible. Indeed, from \eqref{16} it follows that $U_\eta$ is an eigenfunction of $-\Delta_D+r-\eta+\alpha_1 U_\eta$ corresponding to the eigenvalue $0$, that is, $\lambda_1 \big(r-\eta+\alpha_1 U_\eta\big)=0$. But then, by the monotonicity of the principal eigenvalue \cite[Thm.16]{AmannPrinc},
$$
\lambda_1\big(r-\eta+2\alpha_1 U_\eta\big)>\lambda_1\big(r-\eta+\alpha_1 U_\eta\big)=0
$$ 
and so $0$ belongs to the resolvent set of the operator $-\Delta_D+r-\eta+2\alpha_1 U_\eta$. 

We set $\mathcal{R}:= \big(-\Delta_D+r-\eta+2\alpha_1 U_\eta\big)^{-1}$.

\begin{lem}\label{L5}
Let $\Psi_1$ be a strongly positive eigenfunction to the principal eigenvalue $\xi_0=\xi_0(\eta)$ from \eqref{A} and let $\Phi_1:=\alpha_2 \mathcal{R}(P_\eta\Psi_1)$. Then $\mathrm{dim}(\mathrm{ker}(L))=\mathrm{codim}(\mathrm{rg}(L))=1$. In fact, $\mathrm{ker}(L)=\mathrm{span}\{(z_1^*,z_2^*)\}$, where
$$
z_1^*:=Z_1\left(-\alpha_1\Phi_1 u_\eta+\alpha_2\Psi_1p_\eta\,,\,\eta\Phi_1\right)\in \E_1 \ ,\quad z_2^*:=Z_2\big(0,\xi_0(\eta)\Psi_1\big)\in \E_1\ .
$$
\end{lem}

\begin{proof}
For $(\phi,\psi)\in\E_1\times\E_1$ set
\bqnn
T(\phi,\psi):=\left(\begin{array}{c} Z_1\left(-\alpha_1\Phi u_\eta +\alpha_2\Psi p_\eta\,,\,\eta \Phi\right)\\
Z_2\left(0, \xi_0\Psi\right)\end{array}\right) \ 
\eqnn
using convention \eqref{17}.
Since $\Phi,\Psi$ belong to $\Wqb^2$ which is compactly embedded in $L_q$, it is immediate by definition of the operators $Z_1$, $Z_2$ that $T\in\ml(\E_1\times\E_1)$ is compact. Suppose now that $(\phi,\psi)\in\mathrm{ker}(L)$. Then
\begin{align}
\partial_a \phi-\Delta_D \phi&=-\alpha_1U_\eta\phi-\alpha_1\Phi u_\eta+\alpha_2\Psi p_\eta\ , &\phi(0)=\eta \Phi\ ,\label{177}\\
\partial_a \psi-\Delta_D \psi&=\beta_2P_\eta\psi\ , &\psi(0)=\xi_0 \Psi\ ,\label{18}
\end{align}
whence
\begin{align}
(r-\eta)\Phi-\Delta_D \Phi+2\alpha_1U_\eta\Phi -\alpha_2P_\eta\Psi =0\ ,\label{20}\\
(s-\xi_0)\Psi-\Delta_D \Psi-\beta_2P_\eta\Psi=0\ .\label{21}
\end{align}
Since $\xi_0$ is a simple eigenvalue of $-\Delta_D+s-\beta_2P_\eta$, \eqref{21} implies that there is some $\kappa\in\R$ with \mbox{$\Psi=\kappa\Psi_1$}, and thus, by \eqref{20}, $\Phi=\kappa\Phi_1$. From \eqref{177}, \eqref{18} we then derive that $\mathrm{ker}(L)\subset\mathrm{span}\{(z_1^*,z_2^*)\}$. Conversely, let $(\phi,\psi):=(z_1^*,z_2^*)$. Then
\bqn\label{22}
\partial_a \phi-\Delta_D \phi=-\alpha_1 U_\eta\phi-\alpha_1\Phi_1u_\eta+\alpha_2\Psi_1 p_\eta\ , \quad\phi(0)=\eta \Phi_1\ ,
\eqn
and, on integrating with respect to $a$, we obtain
\bqn\label{23}
-\eta\Phi_1+r\Phi-\Delta_D \Phi+\alpha_1U_\eta\Phi =-\alpha_1U_\eta\Phi_1+\alpha_2P_\eta\Psi_1\ . 
\eqn
Clearly, $\Phi=\Phi_1$ solves \eqref{23} and if there be another solution, let $\hat{\Phi}$ denote the difference of the two solutions. Then
$$
-\Delta_D \hat{\Phi}+r\hat{\Phi}+\alpha_1U_\eta\hat{\Phi} =0\ ,
$$
from which
$$
\int_\Om\big\vert\nabla\hat{\Phi}\big\vert^2\,\rd x+r\int_\Om\hat{\Phi}^2\,\rd x +\alpha_1\int_\Om U_\eta\,\hat{\Phi}^2\,\rd x=0
$$
and so $\hat{\Phi}\equiv 0$ (alternatively, we could have invoked \eqref{16} and the monotonicity of the principal eigenvalue). Thus $\Phi=\Phi_1$ is the unique solution to \eqref{23}. Similarly, from the equation satisfied by $\psi=z_2^*$ it follows on integration that
\bqn\label{24}
-\xi_0\Psi_1+s\Psi-\Delta_D\Psi-\beta_2P_\eta\Psi=0\ ,
\eqn
which has the solution $\Psi=\Psi_1$. If $\hat{\Psi}$ denotes the difference to another solution, then
$$
s\hat{\Psi}-\Delta_D \hat{\Psi}-\beta_2P_\eta \hat{\Psi}=0
$$
implying $\hat{\Psi}\equiv 0$ since $\lambda_1(s-\beta_2 P_\eta)>0$. Thus, $\Psi=\Psi_1$ is the unique solution to \eqref{24}, and we conclude that $(z_1^*,z_2^*)\in \mathrm{ker}(L)$. In particular, we have shown that 
\bqn\label{25}
\int_0^\infty e^{-sa}\, z_2^*\,\rd a=\Psi_1\ .
\eqn
Finally, since $\mathrm{dim}(\mathrm{ker}(L))=1$ and $L=1-T$ with a compact operator $T$, the assertion follows.
\end{proof}

It remains to check the transversality condition of \cite{CrandallRabinowitz}.

\begin{lem}\label{L6}
We have $F_{t,(w,v)}(0,0,0)[z_1^*,z_2^*]\not\in\mathrm{rg}(L)$.
\end{lem}

\begin{proof}
From \eqref{17}, \eqref{26}, \eqref{25}, and Lemma~\ref{L4} it follows
$$
F_{t,(w,v)}(0,0,0)[z_1^*,z_2^*]=
\left(\begin{array}{c} 0\\
-Z_2\left(0, \Psi_1\right)\end{array}\right)\ .
$$
Suppose then to the contrary that the assertion is false. Then, by \eqref{27}, there is some $\psi\in\E_1$ satisfying \mbox{$\psi-Z_2(0,\xi_0\Psi)=-Z_2(0,\Psi_1)$}, that is,
$$
\partial_a\psi-\Delta_D\psi-\beta_2P_\eta\psi=0\ ,\quad \psi´(0)=\xi_0\Psi-\Psi_1\ .
$$
Integration with respect to $a$ and testing the resulting elliptic equation with $\Psi_1$ yields
\bqnn
\begin{split}
0&=(s-\xi_0)\int_\Om\Psi\Psi_1\,\rd x+\int_\Om\Psi_1^2\,\rd x-\int_\Om\Psi_1\Delta_D\Psi\,\rd x-\beta_2\int_\Om P_\eta\Psi\Psi_1\,\rd x\\
&=\int_\Om\Psi\big((s-\xi_0)\Psi_1 -\Delta_D\Psi_1 - \beta_2 P_\eta\Psi_1\big)\,\rd x+ \int_\Om\Psi_1^2\,\rd x=\int_\Om\Psi_1^2\,\rd x\ ,
\end{split}
\eqnn
contradicting the positivity of $\Psi_1$. 
\end{proof}

Recall that $\xi_0(\eta)$ is the first eigenvalue of $-\Delta_D+s-\beta_2 P_\eta$. If $\xi$ is regarded as bifurcation parameter in \eqref{4a}-\eqref{5b}, then we obtain in summary the following result:

\begin{thm}\label{T1}
Let $\alpha_j$, $\beta_j$, and $m$ be positive.
 
(a) Besides the trivial solutions $(\xi,u,v)=(\xi,0,0)$ there is a semi-trivial branch of nonnegative classical solutions
$
\mathfrak{C}=\{(\xi,0,v_\xi)\,;\,\lambda_1+s<\xi<\lambda_1+s+\ve_0\}
$
for some $\ve_0>0$, where $v_\xi\not\equiv 0$. There is no nonnegative solution $(\xi,u,v)$ with $u\not\equiv 0$ if $\eta\le\lambda_1+r$.

(b) There is some $\ve_0'>0$ such that, if $\eta\in (\lambda_1+r,\lambda_1+r+\ve_0')$, then in addition to $\mathfrak{C}$ there is another semi-trivial branch $\mathfrak{C}_\eta=\{(\xi,u_\eta,0)\,;\,\xi\ge 0\}$ of nonnegative classical solutions to \eqref{4a}-\eqref{5b},
where $(\eta,u_\eta)\not\equiv (\eta,0)$ solves \eqref{4a}, \eqref{4b} with $V\equiv 0$. Moreover, provided that $\frac{\beta_2(\eta-r)}{\alpha_1+m(\eta-r)}$ is sufficiently small, in particular less than $\lambda_1+s$, local supercritical bifurcation of a branch of positive classical solutions occurs at the critical point $(\xi_0(\eta),u_\eta,0)\in \mathfrak{C}_\eta$. That is, there are $\ve_\eta>0$ and a branch of solutions 
$$
\mathfrak{C}_*:=\{(\xi,u_*,v_*)\,;\,\xi_0(\eta)<\xi<\xi_0(\eta)+\ve_\eta\}
$$
with $(u_*,v_*)\ge 0$ and $u_*\not\equiv 0$, $v_*\not\equiv 0$.
\end{thm}

\begin{proof}
Part (a) is a consequence of Lemma ~\ref{L1} and Lemma~\ref{L2}. For (b) we fix $(\eta,u_\eta)\in\mathcal{U}$ as before and consider a solution $(u,v)=(u_\eta-w,v)$. Then Lemma~\ref{L5}, Lemma~\ref{L6}, and \cite[Thm.1.7]{CrandallRabinowitz} imply that $(\xi_0(\eta),0,0)$ is a bifurcation point of \eqref{12}, \eqref{13} and close to this point the nontrivial solutions $(w,v)$ lie on the curve (for some $\ve_\eta>0$)
$$
\big(\xi(\ve),\ve z_1^*+\ve\Theta_1(\ve),\ve z_2^*+\ve\Theta_2(\ve)\big)\ ,\quad \vert\ve\vert<\ve_\eta\ ,
$$
where $\xi:(-\ve_\eta,\ve_\eta)\rightarrow \R$ is continuous with $\xi(0)=\xi_0(\eta)$ and $\Theta=(\Theta_1,\Theta_2):(-\ve_\eta,\ve_\eta)\rightarrow\E_1\times\E_1$ is continuous with $\Theta(0,0)=(0,0)$.
Therefore, in terms of $(u,v)$ we obtain that $(\xi_0(\eta),u_\eta,0)$ is a bifurcation point of \eqref{4a}-\eqref{5b} and close to this point the solutions lie on the curve
$$
\big(\xi(\ve),u_\eta-\ve z_1^*-\ve\Theta_1(\ve),\ve z_2^*+\ve\Theta_2(\ve)\big)\ ,\quad \vert\ve\vert<\ve_\eta\ .
$$
Let $\ve\in (0,\ve_\eta)$ be fixed and set $u_*:=u_\eta-\ve z_1^*-\ve\Theta_1(\ve)$ and $v_*:=\ve z_2^*+\ve\Theta_2(\ve)$. Then, by definition of $z_j^*$,
$$
u_*(0)=u_\eta(0)-\ve\eta\Phi_1-\ve\gamma_0\Theta_1(\ve)\ ,\qquad
v_*(0)=\ve\xi_0(\eta)\Psi_1+\ve\gamma_0\Theta_2(\ve)\ .
$$
Clearly, $\Psi_1$ belongs to the positive cone of $\Wqb^{2-2/q}$ and thus, since $\gamma_0\Theta_2\in C\big((-\ve_\eta,\ve_\eta),\Wqb^{2-2/q}\big)$ and $\xi_0(\eta)>0$, we have $v_*(0)\ge 0$ provided that $\ve>0$ is sufficiently small. This yields that $v_*$ satisfies \eqref{5a}, \eqref{5b} and is positive. As for the positivity of $u_*$ we note that $u_\eta(0)=\eta U_\eta$ with $U_\eta$ being strongly positive and so is $u_\eta(0)$.  Thus, if $\ve>0$ is sufficiently small, we deduce the positivity of $u_*(0)$, whence of $u_*$ by \eqref{4a}, \eqref{4b}. That necessarily $\xi>\xi_0(\eta)$ was shown in Lemma~\ref{L3}. Finally, standard regularity theory for semilinear parabolic equations implies that both $u_*$, $v_*$ are classical solution to \eqref{4a}-\eqref{5b}, i.e. $u_*,v_*$ belong to $C(\R^+\times\bar{\Om})\cap C^{1,2}((0,\infty)\times\bar{\Om})$.
\end{proof}

\begin{rem}
We shall point out that while global bifurcation results are shown in \cite{BlatBrown} for \eqref{6}, \eqref{7}, our bifurcation results for \eqref{4a}-\eqref{5b} are of purely local character. This is due to a lack of compactness of, e.g., the map $\E_1\times\E_1\rightarrow \E_0, (u,v)\mapsto Uv$ with respect to the age variable $a$.
\end{rem}

\subsection{Bifurcation for the Parameter $\eta$}

We now consider $\eta$ as bifurcation parameter in \eqref{4a}-\eqref{5b} and keep $\xi$ fixed. Let $(\xi,v_\xi)\in\mathcal{V}$ from Lemma~\ref{L2} be fixed and set $V_\xi:=\int_0^\infty e^{-sa}\,v_\xi(a)\rd a$. Then there is a branch of semi-trivial solution
$$
\mathcal{D}_\xi:=\big\{(\eta,0,v_\xi)\,,\,\eta\ge 0\big\}\ .
$$
The goal is to prove that bifurcation of positive solutions occurs from this branch. Since the idea is exactly the same as in the previous subsection, we merely sketch the proof and omit details. Proceeding as before we suppose that $(\eta,u,v)=(\eta_0+t,u,v_\xi+w)$ solves \eqref{4a}-\eqref{5b} with $\eta_0=\eta_0(\xi)$ to be determined. Then
the analogues to \eqref{12}, \eqref{13} read
\begin{align}
\partial_a u-\Delta_D u&=-\alpha_1Uu-\alpha_2\frac{(V_\xi+W)u}{1+mU}\ ,&u(0)=(\eta_0+t) U\ ,\label{100}\\
\partial_a w-\Delta_D w&=-\beta_1Wv_\xi-\beta_1(V_\xi+W)w+\beta_2\frac{U(v_\xi+w)}{1+mU}\ ,&w(0)=\xi W\ ,\label{101}
\end{align}
where
\bqnn
U:=\int_0^\infty e^{-ra}\, u(a)\,\rd a\ ,\qquad  W:=\int_0^\infty e^{-sa}\,w(a)\,\rd a\ .
\eqnn
As in Lemma~\ref{L4} we derive that the operators $-\Delta_D+\alpha_2 V_\xi$ and  $-\Delta_D+\beta_2 V_\xi$ have maximal $L_q$-regularity on $\R^+$, i.e.,
\begin{align*}
&S_1:=\left(\partial_a-\Delta_D+\alpha_2 V_\xi,\gamma_0\right)^{-1}\in\ml(\E_0\times \Wqb^{2-2/q},\E_1)\ ,\\
&S_2:=\left(\partial_a-\Delta_D+\beta_1 V_\xi,\gamma_0\right)^{-1}\in\ml(\E_0\times \Wqb^{2-2/q},\E_1)\ 
\end{align*}
are well-defined (note that we do not impose any restriction on the coefficients in this case). Thus, solutions of 
\eqref{100}, \eqref{101} are the zeros of
$$
G(t,u,w):=\left(\begin{array}{c} u-S_1\left(\alpha_2V_\xi u-\alpha_1Uu-\alpha_2\frac{(V_\xi+W)u}{1+mU}\,,\,(\eta_0+t) U\right)\\
w-S_2\left(-\beta_1W(v_\xi+w)+\beta_2\frac{U(v_\xi+w)}{1+mU}\,,\, \xi W\right)\end{array}\right)\ .
$$
Linearizing around $(t,u,w)=(0,0,0)$ gives for $(\phi,\psi)\in\E_1\times\E_1$ with \eqref{17}:
\bqnn
G_{(u,w)}(0,0,0)[\phi,\psi]=\left(\begin{array}{c} \phi-S_1\left(0, \eta_0\Phi\right)\\
\psi-S_2\left(\beta_2\Phi v_\xi-\beta_1\Psi v_\xi\,,\, \xi\Psi\right)\end{array}\right)\ 
\eqnn
and
\bqnn
G_{t,(u,w)}(0,0,0)[\phi,\psi]=\left(\begin{array}{c} -S_1\left(0, \Phi\right)\\
0\end{array}\right)\ .
\eqnn
Thus, if $(\phi,\psi)\in\mathrm{ker}(\tilde{L})$ with $\tilde{L}:=G_{(u,w)}(0,0,0)$, then
$$
\phi=S_1(0,\eta_0\Phi)\ ,\quad \psi=S_2(\beta_2\Phi v_\xi-\beta_1\Psi v_\xi\,,\, \xi \Psi)
$$
and, on integrating with respect to $a$,
$$
-\Delta_D\Phi+\alpha_2 V_\xi\Phi=(\eta_0-r)\Phi\ ,\qquad -\Delta_D\Psi+(2\beta_1 V_\xi-\xi+s)\Psi=\beta_2 V_\xi \Phi\ .
$$
But then, if $\eta_0=\eta_0(\xi)$ is the principal eigenvalue of $-\Delta_D+r+\alpha_2 V_\xi$ and $\tilde{\Phi}_1$ a corresponding strongly positive eigenfunction, then we derive as in the proof of Lemma~\ref{L5} that the kernel of $\tilde{L}$ is one-dimensional and spanned by $(s_1^*,s_2^*)\in \E_1\times\E_1$, where
$$
s_1^*:=S_1(0,\eta_0\tilde{\Phi}_1)\ ,\qquad s_2^*:=S_2(-\beta_1\tilde{\Psi}_1v_\xi+\beta_2\tilde{\Phi}_1 v_\xi\,,\,\xi\tilde{\Psi}_1)\ ,
$$
and $\tilde{\Psi}_1:=\big(-\Delta_D+2\beta_1 V_\xi-\xi+s\big)^{-1}(\beta_2 V_\xi \tilde{\Phi}_1)$. Also, the codimension of the range of $\tilde{L}$ equals one. Analogously to the proof of Lemma~\ref{L6} we deduce that
$$
G_{t,(u,w)}(0,0,0)[s_1^*,s_2^*]=
\left(\begin{array}{c} 
-S_1\big(0, \tilde{\Phi}_1\big)\\
0\end{array}\right)\
$$
does not belong to the range of $\tilde{L}$. Therefore, we are again in a position to apply \cite[Thm.1.7]{CrandallRabinowitz}. Recalling Lemma~\ref{L1} and Lemma~\ref{L2} we obtain the following analogue of Theorem~\ref{T1} for bifurcation with respect to the parameter $\eta$:

\begin{thm}\label{T2}
Let $\alpha_j$, $\beta_j$, and $m$ be positive.

(a) Besides the trivial solutions $(\eta,u,v)=(\eta,0,0)$ there is a semi-trivial branch of nonnegative classical solutions
$
\mathfrak{D}=\{(\eta,u_\eta,0)\,;\,\lambda_1+r<\eta<\lambda_1+r+\ve_0'\}
$
for some $\ve_0'>0$, where $u_\eta\not\equiv 0$. There is no nonnegative solution $(\eta,u,v)$ with $v\not\equiv 0$ if $\xi\le\lambda_1+s-\beta_2/m$.

(b) There is some $\ve_0>0$ such that, if $\xi\in (\lambda_1+s,\lambda_1+s+\ve_0)$, then in addition to $\mathfrak{D}_1$ there is another semi-trivial branch $\mathfrak{D}_\xi=\{(\eta,0,v_\xi)\,;\,\eta\ge 0\}$ of nonnegative classical solutions to \eqref{4a}-\eqref{5b},
where $(\xi,v_\xi)\not\equiv (\xi,0)$ solves \eqref{5a}, \eqref{5b} with $U\equiv 0$. Moreover, a local branch of positive classical solutions $(\eta,u_*,v_*)$ bifurcates from the critical point $(\eta_0(\xi),0,v_\xi)\in \mathfrak{D}_\xi$
with $(u_*,v_*)\ge 0$ and $u_*\not\equiv 0$, $v_*\not\equiv 0$.
\end{thm}


\end{document}